\documentclass[11pt,amsmath,amsthm,amsfonts,amssymb,amscd]{article}
\usepackage{epsfig}
\usepackage[active]{srcltx}
\usepackage[latin1]{inputenc}
\usepackage{amsthm}
\usepackage{amsmath}
\usepackage{graphicx}
\usepackage{psfrag}
\font \Bbbten=msbm10 \font \Bbbsev=msbm7 \font \Bbbfiv=msbm5
\newfam\Bbbfam \def \Bbb{\fam\Bbbfam\Bbbten}\textfont\Bbbfam = \Bbbten
\scriptfont\Bbbfam=\Bbbsev \scriptscriptfont\Bbbfam=\Bbbfiv
\newcommand{\N}{\mbox{$I\!\!N$}}
\newcommand{\Z}{\mbox{$Z\!\!\!Z$}}

\newcommand{\R}{\mbox{$I\!\!R$}}

\newcommand{\T}{{\Bbb T}}

\newcommand{\dist}{{\rm dist}}

\theoremstyle{plain}
\newtheorem{thm}{Theorem}[section]

\newtheorem{maintheorem}{Theorem}
\newtheorem{claim}[thm]{Claim}

\newtheorem{Df}{Definition}[section]
\newtheorem{Teo}{Theorem}[section]
\newtheorem{Lem}[Teo]{Lemma}
\newtheorem{Cor}[Teo]{Corollary}

\newtheorem{Obs}[Teo]{Remark}

\title{{ Exponential speed of mixing for skew-products with singularities}}
\author{R. Markarian\footnote{partially supported by Proyecto PDT 2006-2008. S/C/IF/54/001, Uruguay}, M. J. Pacifico\footnote{partially supported by CNPq Brazil, Pronex on Dynamical Systems,
FAPERJ}, J. L. Vieitez\footnote{partially supported by PEDECIBA, ANII, Uruguay}}
\date{\today}

\begin{document}

\maketitle

\begin{abstract}

Let $f: [0,1]\times [0,1] \backslash \{ 1/2\} \to
[0,1]\times [0,1]$ be the $C^\infty$ endomorphism given by
$$f(x,y)=\left( 2x-\lfloor 2x\rfloor, \, y+ \frac{c}{|x-1/2|}-\left\lfloor y+
\frac{c}{|x-1/2|} \right\rfloor \right),\; c\in \R^+$$
 We prove that $f$ is topologically mixing and if $c>1/4$ then
$f$ is mixing with respect to Lebesgue measure. Furthermore we prove that the
speed of mixing is exponential.
\end{abstract}

\thanks{\footnotesize{


}}

\setcounter{tocdepth}{2}

\tableofcontents

\section{Introduction}

A basic problem in dynamics is the understanding of the ergodic behavior
of a given dynamical system.
Frequently this is translated into the knowledge of mixing properties of
the system.
Once mixing is established it is natural to ask for the rate or speed of mixing of the system.

For hyperbolic systems and nonuniform hyperbolic ones , without or with singularities, this kind of study is well
understood and the techniques to do so
have been developed by several authors.
We indicate the works by Sinai \cite{Si}, Pesin \cite{Pe}, Pesin and Barreira,  \cite{BP}, Dolgopyat \cite{Do}, Ruelle \cite{Ru}, Bowen \cite{Bo}, L-S Young \cite{Yo1}\cite{Yo2}, Benedicks \cite{BY}, Baladi \cite{Ba}, Viana \cite{Vi}, Walters \cite{Wa}, and the references therein to the interested reader.

When the system  $T$ under study has singularities, the phase space is not the whole manifold
and in this case one asks zero-Lebesgue measure for the union $\cup_{n=0}^\infty T^{-n}S$ of
the set of singularities $S$.
This is the case of billiards, studied by Sinai, Chernov \cite{Ch},\cite{CY}, Markarian \cite{CM}, Bunimovich \cite{Bu} and others.
In these cases we have the additional difficulty that the stable and unstable manifolds of points may be arbitrarily short since their length is conditioned by the distance of the points to $S$.

In general, the presence of singularities adds complexity into the problem
and makes the analysis much more difficult. Nevertheless, in this paper,
where we study a certain skew-product with  singularities on the fiber,  it is
 the presence of singularities, jointly with the expanding action in the
base that enable us to obtain all the chaotic behavior of the system.

\medbreak

In this paper we are interested in the mixing properties of the skew-product\footnote{Recall that a skew-product
$T$ is an automorphism of the measure space $X\times Y$
where $X$ and $Y$ are measure spaces and the action of $T:X\times Y\to X\times Y$  has the form
$$T(x,y)=(A(x),B_x(y))\, ;\quad x\in X,\, y\in Y\, ,$$
where $A$ is an automorphism of the space $X$ (the "base" ) and $B_x(y)$ ,
with $x$ fixed, is an automorphism of $Y$ (the "fiber" ).
The concept of a skew-product extends directly to the case of endomorphisms.}
given by the
$C^\infty$-endomorphism $f: [0,1]\times \big([0,1] \backslash \{
1/2\}\big) \to [0,1]\times [0,1]$ defined by
$$f(x,y)=\left( 2x-\lfloor 2x\rfloor, \, y+ \frac{c}{|x-1/2|}-\left\lfloor y+
\frac{c}{|x-1/2|} \right\rfloor \right),\; c\in \R^+.$$
Here, given a real number $x$, $\lfloor x \rfloor$ stands for the greatest
integer less or equal to $x$.
\medbreak

Since the denominator $\frac{c}{|x-1/2|}$ vanishes at $x=1/2$, the line
$\{(1/2,y)\,:\; y\in [0,1)\}$ is constituted  by singularities of $f$.
Besides that, for $c\neq 0$ we have
that the vertical projection of $f(x,y)$ sharply varies when $x\approx 1/2$.

Identifying
$[0,1]\times [0,1]$ with the two-dimensional torus $\T^2$,
the skew-product may be seen as
 defined in $\T^2$ where the circle given by
$x=1/2$ is a curve of singularities of $f$.

The successive
iterates by $f$ of a rectangle $R$ are
transformed into a denumerable set of strips
accumulating onto the circle $x=1$ in the torus.
This effect together with the fact that the pre-orbit by $x\mapsto 2x\mod(1)$ of the circle
$x=1/2$ is dense in the torus are responsible of the rich chaotic dynamics
observed in this system.

Since the length of vertical segments are preserved under $f$,
the action of $f$ on the vertical borders  of  $R$ is just a translation depending
continously on $x \in [0,1]\setminus \{1/2\}$.
Hence, the stretching and accumulation of
the iterates of $R$ onto the pre-orbit of the circle $x=1/2$ in the torus
is due to the slipping effect of $f$ in the horizontal borders of $R$.

The skew-product $f$ can be also immersed in a one-parameter family
of expanding skew-products with the same line of singularities:
$$ f_\lambda(x,y)= (2x, \lambda y + \frac{c}{|x-1/2|}), \,\, \lambda \geq 1\, .$$
Thus, it is interesting to detected the ergodic properties in the limit dynamics
given by $\lambda=1$. For instance, transitivity, mixing and rate of mixing.

In this paper we prove that the skew-product $f$ is topologically mixing,
preserves the Lebesgue measure $m$ on the torus, is mixing with respect to
$m$ and finally we prove that the rate of mixing is exponential.

\subsection{Toy model of flows with  a singularity: slipping effect.}
Let $M$ be a $3$-dimensional manifold and assume that $\Phi:M\to M$
is a flow containing a transitive attractor
$\Lambda\subset M$ with a hyperbolic singularity $p\in \Lambda$.
 The geometric Lorenz attractor and any Lorenz-like attractor satisfy these
 conditions, see \cite{GW79,Lo63,AraPac2010}.

 We consider the case when the singularity has three real eigenvalues
$\lambda_i$,
$1\leq i \leq 3$, and satisfy
$\lambda_2(\sigma)<\lambda_3(\sigma)<0<-\lambda_3(\sigma)
<\lambda_1(\sigma)$.
 Via Hartmann-Gro{\ss}man
 theorem we assume that we have linearized coordinates in a neighborhood
 $U \supset [-1,1]^3$ of the singularity $p$ in such a way that $\lambda_1$
 corresponds to $0x$-axis, $\lambda_2$ to $0y$-axis and $\lambda_3$ to $0z$-axis.

Let $S=\{(x,y,z)\in U\,:\,z=1\}$ be a transverse section to the flow so
that every trajectory eventually
crosses $S$ in the direction of the negative $z$-axis.
Consider also $\Sigma = \{(x, y, z) : |x| = 1\} = \Sigma^+\cup\Sigma^-$
 with $\Sigma^\pm = \{(x, y, z) : x = \pm 1\}\, .$
 To each $(x_0, y_0, 1) \in  S$
 the time $\tau$ such that $X^\tau (x_0, y_0, 1) \in \Sigma$ is given by
 $\tau(x_0) = \frac{-1}{\lambda_1}\log |x_0|$, and it is such that $\tau(x_0)\to\infty$ when
$x_0 \to 0$.
This fact has the effect that  different slices parallel to $0y$-axis of the section $S$
arrives to $\Sigma$ with a delay. Hence,  we cannot see the return of each slice to $S$
at the same time, even when the expecting delay is bounded .

Assume now that we "forget" the effect of the singularity and consider that the
return time is the same for points in a same slice.
Also "forget" the strong stable direction.
Note that the strong stable direction does not interfere in the dynamics of the geometric Lorenz attractor.

After these identifications, the dynamics in a neighborhood of $p$ occurs in the $(x,z)$ plane, and
 may be seen as a slipping in the vertical direction in order to annihilate the delay of time.
 Since the delay goes to infinity as $x\to 0$ the slipping also goes to infinity when $x\to 0$.
Thus, the dynamics there is given by  $(\phi(x),\psi(x,z))$ with $\psi(x,z)\to\infty$ when $x\to 0$.
Moreover, since the ratio $ \beta=-\frac{\lambda_2}{\lambda_1}$ is greater than one, the dynamics in
the $x$ direction is expanding.

Thus, changing the name of the variable $z$ by $y$, the skew-product
$$f(x,y)=\left(2x\mod (1), y+\frac{c}{|x-1/2|}\mod (1)\right)$$ may be seen as a simplified case of the
slipping effect in singular hyperbolic attractors, as is the case of a Lorenz-like attractor.

\subsection{Statement of results.}
To announce in a precise way our results let us introduce some definitions and
 related facts proved elsewhere.

\begin{Df}
Let $(X, \mathcal{A}, f, \mu)$ be a dynamical system defined on the
space $X$, $\mathcal{A}$ a 
$\sigma$-algebra of
$X$, and $\mu$ an $f$-invariant probability measure.
The map $f$ is mixing  if for all pair of sets
$A,B\in\mathcal{A}$, we have
$$\lim_{n\to\infty} \mu(f^{-n}(A)\cap B)=\mu(A)\mu(B)\,.$$
\end{Df}

A form of mixing that can be defined without appealing to measures
is the following
\begin{Df}
Let $f:X\to X$ be a continuous map defined in the topological space
$X$. We say that the dynamical system defined by $f$ is
topologically mixing if for every pair of non-empty open subsets
$A,\, B$ of $X$ there is $N>0$ such that $\forall n\geq N: \,
f^n(A)\cap B\neq\emptyset$.
\end{Df}
There is even a commonly used weaker notion: we say that the
system defined by $f$ is topologically transitive if for every pair
of non-empty open subsets $A,\, B$ of $X$ there is $n\in\Z $ such
that $f^n(A)\cap B\neq \emptyset$.

It is well known that if a dynamical system is defined on topological
space $X$, $\mathcal{A}$ is the Borel $\sigma$-algebra of
$X$ and $\mu$ is a probability invariant measure such that $\mu(A) > 0$
for every open set $A$ of $X$, then if the system is mixing it is topological
mixing.
This and other general results on Ergodic Theory may be
found in \cite{Wa}, for instance.
\vspace{0.2cm}

The main results in this paper are:

\begin{maintheorem}
\label{Teo1}
For all positive $c$ the skew-product $f:\T^2\to\T^2$ is topologically mixing.
 \end{maintheorem}

\begin{maintheorem}
 \label{Teo2}
The skew-product $f$ preserves the Lebesgue measure $m$ in the torus and, for $c> 1/4$,
$f$ is mixing with respect to $m$.
\end{maintheorem}

\begin{maintheorem}
 \label{Teo3}
The rate of mixing is exponential, that is, there is $0< \lambda < 1$ such that
for all pair of sets $A$ and $B$ we have
$$
|m(f^{-n}(A)\cap B) - m(A) m(B)| < \lambda^n m(A) m(B), \quad \mbox{ for all} \quad n \geq 0\,.
$$
\end{maintheorem}

Next we list two interesting features of the skew-product $f$
\begin{itemize} \label{interesting}
\item[$(\star)$] \,\, {\em{For all $p=(x,y)\in \T^2$,  there is no stable manifold $W^s(p,f)$.}}\/
Indeed, given $(x,y)$ and $(x+\Delta x,y+\Delta y)$, assuming that
$x<1/2$, $\Delta x\neq 0$, $x+\Delta x<1/2$ and  computing
$$\|f(x+\Delta x, y+\Delta y)-f(x,y)\|=\left\|2\Delta x,
\Delta y+\frac{c}{1/2-x}\left[\frac{1}{1-\frac{\Delta x}{1/2-x}}-1\right]\mod 1\right\|\geq$$
$$\geq 2|\Delta x|,\;\mbox{ and similar result holds for }\;x>1/2\, .$$
Hence, if $\Delta x\neq 0$,
$\dist(f^n(x+\Delta x,y+\Delta y),f^n(x,y))\geq 2^n\Delta x\mod 1$ which does not converges to 0.
On the other hand, if $\Delta x=0$ then the distance between $f^n(x,y+\Delta y)$ and $f^n(x,y)$ is preserved.
Thus, for no $(\Delta x, \Delta y)$ we have $\dist(f^n(x+\Delta x,y+\Delta y),f^n(x,y))\to 0$ when $n\to+\infty$.

\item[$(\star \star)$] \,\,   {\em {The unstable manifolds are not unique.}}\/ Indeed,  for any itinerary $\{(x_n,y_n)\}_{n\in\N}$ such that
$f(x_n,y_n)=(x_{n-1},y_{n-1})$ it is defined
an unstable manifold $W^u((x_0,y_0),f)$ (recall that $f$ is an endomorphism) and so the unstable
manifold of a point is not unique.
Moreover, $f$ is not an expanding map since for any $p=(x,y)$ we have  $Df_p(0,1)=(0,1)$.
Finally, it has no dominated splitting (see Section \ref{generalsetting} for the proof of these facts).

\end{itemize}
Thus, the standard techniques in dynamics using
existence of stable and unstable manifolds, for instance, are useless
here.

\section{Preliminaries } \label{generalsetting}

In this section we establish some preliminaries properties of $f$ that
will be used in the proofs.
We identify the set $Q=[0,1]\times [0,1]$ with the 2-torus $\T^2$.

If $0\leq x<1/2$ then we have that $f(x,y)=(2x,
y+\frac{c}{1/2-x}-\lfloor y+\frac{c}{1/2-x}\rfloor)$ while if
$1/2<x< 1$ then $f(x,y)=(2x-1,y+\frac{c}{x-1/2}-\lfloor
y+\frac{c}{x-1/2}\rfloor)$.

The matrix $\left[Df_{(x,y)}\right]$ is in the case $0\leq x<1/2$ given by
$\left(
  \begin{array}{cc}
    2 & 0 \\
    \frac{c}{(1/2-x)^2} & 1 \\
  \end{array}
\right) $
and in the case $1/2<x\leq 1$ by $\left(
    \begin{array}{cc}
      2 & 0 \\
      \frac{-c}{(1/2-x)^2} & 1 \\
    \end{array}
  \right)$. Therefore it depends only on $x$. Any vector different
  of a vertical one is expanded by the action of $Df$ which presents
  two eigenvalues: $1$ with eigenvector $(0,1)$,
  and $2$ with eigenvector
  $(1,\frac{c}{(x-1/2)^2})$ if $0\leq x<1/2$ and
  $(-1,\frac{c}{(x-1/2)^2})$ if $1/2<x<1$.
  Hence we have no stable manifold at any point of $\T^2$ (see $(\star)$\, )
  and points at the left of the line $x=1/2$ have eigenvectors corresponding to the eigenvalue 2
   forming an acute angle with the $Ox$ axis such that when $x\to
   1/2$ the angle between the eigenvector associated to 2 tends to
   be vertical.
   A similar picture is
   valid at points at the right of $x=1/2$ taking into account
   that in that case the eigenvector associated to 2 forms an
   obtuse angle with the $Ox$ axis. From these facts one may see that no non-trivial splitting is preserved. Indeed, given a periodic orbit, no direction different of the vertical one is preserved.

Given a real number $a  \in (0,1)$, we write
$$
a = \sum_1^\infty \frac{a_i} {2^i}, \; \; \; a \sim 0.a_1\cdots a_n\cdots\quad a_j\in (0,1)
$$ for its binary decomposition.

Writing $x\in [0,1)$ in base $2$ the dynamics in
the $x$- coordinate
is as the shift
$$\sigma:\{0,1\}^{\N}\to \{0,1\}^{\N},
\quad
\sigma(b_1  b_2 b_3 \cdots  )= b_2 b_3\cdots
$$
Each point $x \sim (b_1  b_2 \cdots  )$ has two pre-images by
this map  \begin{equation} \label{eq:pre}
\sigma^{-1}(b_1  b_2 \cdots )=\left\{\begin{array}{l}
 x_0 \sim (0 \, b_1  b_2 \cdots) \sim x/2 \\
 x_1 \sim (1 \, b_1  b_2 \cdots ) \sim (1+x)/2
             \end{array}\right.
\end{equation}
Since $f(x,y)= (2x,y+c/|x-1/2|)\mod(1)$ equation (\ref{eq:pre})
implies that any  $Z=(x, y) \in \T^2$, with $x=0.b_1b_2\cdots $
has two pre-images $Z_0,\, Z_1$ by $f$ given by:
\begin{enumerate}
 \item[(a)] $Z_0=\left(\frac{x}{2}, y - \frac{2c}{1-x}-\lfloor y-\frac{2c}{1-x}\rfloor\right)=(x_0,y_0),$
\item[(b)] $Z_1=\left(\frac{1+x}{2}, y-\frac{2c}{x}-\lfloor y- \frac{2c}{x}\rfloor\right)=(x_1,y_1).$
\end{enumerate}
Inductively, given a sequence ${{b}}=(b_1b_2\cdots b_n)$ of
{\it length} $ |b| = n$, with $b_j\in\{0,1\},\,\,\forall j\leq n$,
and assuming that $Z_{b_2b_3\cdots b_n}$ (one of the
$(n-1)$-th preimages of $Z$)is already defined we have
that one of the $n$-th preimages of $Z$ is
$Z_{{{b}}}=(x_{{{b}}},y_{{{b}}})$ with
\newpage
\begin{equation} \label{badalo}
(a) \hspace{2.7cm}  x_{{{b}}}=\frac{b_1+ x_{b_2 b_3\cdots
b_n}}{2}\qquad\hspace{3cm}
\end{equation}
and
\begin{equation*}
  (b) \quad y_{{{b}}}=\left(y_{b_2\cdots b_n} -
\frac{2c}{(1-b_1)+(2b_1-1)x_{b_2\cdots b_n}}\right) \mod(1)
\end{equation*}

We remark that if $Z= (x,y), \; \; W=(x',y')$ and
$b= (b_1b_2\cdots b_n)$ then $|x_b -x'_b| = |x -x'|/ 2^n$.
We also remark that
for any $x\in [0,1)$ the set of preimages $\mathcal{S}_n$
of $x$ for all the different $b$'s of length $n$
is almost uniformly distributed in
$[0,1)$, i.e., for any interval $I \subset [0,1)$:
\begin{equation} \label{eq1}
\lim_{n\to \infty} \frac{\#(\mathcal{S}_n\cap I)}{\#\mathcal{S}_n} =
\ell(I) .
\end{equation}
Here $\# X$ means the cardinality of $X$ ($\#S_n = 2^n$),
and $\ell(I)$ is the length of $I$.

We extend this notation to the $n$-th preimage of an horizontal
segment $I = [Z, Y]$: $Z_b(I)$ is the $n$-th pre-image
of $I$ that has $Z_b$ as one of its boundaries. In the same way,
if $R$ is a rectangle whose lower bound is $I$, then $Z_b(R)$ is
the $n$-th pre-image of $R$ with $Z_b(I)$ as one of its ``sides".

\begin{Lem} \label{lema21}
 The vertical projection
$\Pi_y(f(\gamma))$ of the image of a monotone arc $\gamma(t)=(x(t),y(t))$ (i.e.,
an arc such that $x(t)$ and $y(t)$ are monotone functions) whose horizontal
projection $\Pi_x(\gamma)$ has length greater or equal to $2/c$
covers all $[0,1)$.
\end{Lem}
\begin{proof}
Given a monotone arc $\gamma:[0,1]\to \T^2$, $\gamma(t)=(x(t),y(t))$,
the vertical projection of the function
$f(x(t),y(t))=(2x(t) \mod(1),y(t)+\frac{c}{|x(t)-1/2|} \mod(1))$
varies from $y(0)+\frac{c}{|x(0)-1/2|} \mod(1)$ to
$y(1)+\frac{c}{|x(1)-1/2|} \mod(1)\,,$ covering
$|y(1)-y(0)+\frac{c}{|x(1)-1/2|}-\frac{c}{|x(0)-1/2|}| \mod(1)$.

In order to simplify computations we suppose that $\gamma$ is in
$\R^2$.
Then if its horizontal projection includes $k + 1/2, k \in \Z$
the vertical projection has infinite length.
Assume now that
$3/2 > x(1) > x(0) + 2/c > x(0) > 1/2.$
Thus,
$$
\left|\frac{c}{|x(1)-1/2|}-\frac{c}{|x(0)-1/2|}\right|=
\left|\frac{c}{x(1)-1/2}-\frac{c}{x(0)-1/2}\right|> c\frac{x(1)-x(0)}{(x(1)-1/2)(x(0)-1/2)}>4\, .
$$
Thus, since $-1\leq y(1)-y(0)\leq 1$ we have $|y(1)-y(0)+\frac{c}{|x(1)-1/2|}-\frac{c}{|x(0)-1/2|}| >2$.
\end{proof}

\section{The skew-product is topologically mixing}
 \label{topmixing}

Recall that a map $f: M \to M$ is topologically mixing if for all
pair of open sets $A, B $ of $M$ there is $N$ such that for all $n \geq N$
it holds $f^n(A)\cap B\neq\emptyset$.

\begin{Teo} \label{topmix}
If $c\in\R^+$ then $f(x,y)$ is topologically
mixing.
\end{Teo}
\begin{proof}
It is enough to prove the statement for open rectangles $A$ and $B$ of sides
parallel to the coordinate axes since they form a basis for the
standard topology of the plane.

The idea of the proof is as follows:
Let $(x_A,y_A)$ be the coordinates of the center of $A$ and $(x_B,y_B)$ the coordinates of
the center of $B$.
 We pick a suitable pre-image of $(x_B,y_B)$,  $Z_{{{r}}}=(x_{{{r}}},y_{{{r}}})$,
 as in (\ref{badalo}) so that
$Z_{r}$ is close to $(x_A,y_A)$ and such that for some $n$, with $0<n<|r|$, it holds
that $f^n(Z_r)$ is in a small enough neighborhood of $\{x=1/2\}$ to guarantee that
the pre-image of $Z_{r_n}(S)$ (here $r_n$ represents the sub-string of length $n$
contained in $r$), where $S$ is the horizontal segment contained in $B$ and passing
through $(x_B,y_B)$, is almost vertical and has length greater than 1.
This implies that the pre-image $Z_r(S)$ cuts $A$ and
 thus we obtain that $f^{|r|}(Z_r(S))$ cuts $B$.

To begin with the proof let $2\delta_A$ be the length of $\Pi_x(A)$ and $2\delta_B$
the length of $\Pi_x(B)$.
 Let also, in base 2, $x_A=(0.a_1a_2a_3\ldots)_2$ and $x_B=(0.b_1b_2b_3\ldots)_2$ and find $N$
such that for $\delta=\min\{\delta_A,\delta_B\}$ we have $1/2^N<\delta$ and so $2^N\,\delta>1$.
 Now we consider
 $$
 r=0.a_1a_2\ldots a_N0  \underset{N\text{ ones}}{\underbrace{11\ldots 1}}0b_1b_2b_3\ldots b_N\, .
 $$
 Clearly $\Pi_x(Z_r)$ is near $x_A$ and lies in $[x_a-\delta_A,x_A+\delta_A]$ since
 $|x_r-x_A|<1/2^{N}$. Now we choose $y_r\in(0,1)$ such that $\Pi_y(Z_r)$ belongs to $y=y_A$.

After $N$ iterates by $f$ we have that $f^N(Z_r)$ is at a distance less that $1/2^N$ from $\{x=1/2\}$
(since $\Pi_x(f^N(Z_r))=0.0 11\ldots 10b_1b_2b_3\ldots \,$).

It holds that the vertical projection of
$f^N(A)$ has length greater than 1 and $\Pi_x(f^{2N+2}(Z_r))\in \Pi_x(B)$ and $\ell(\Pi_y(f^{2N+2}(A)))>1$ too.
Therefore $f^{2N+2}(A)\cap B\neq\emptyset$.

 Since the length of the
horizontal projection doubles under iterations by the action $x\mapsto 2x\mod 1$, there is $N_1>0$ such that for $n>N_1$
we have that $\Pi_x(f^n(A))$ covers all $[0,1]$.
It follows that $f^{N_1+2N+2+k}(A)\cap B\neq\emptyset$ for all $k\geq 0$.
Thus $f$ is topologically mixing, proving Theorem \ref{Teo1}
\end{proof}

\section{ Lebesgue measure preserved and mixing.}
In this section we prove Theorem \ref{Teo2}.
We start establishing some auxiliary lemmas.
The first says that even for the worst case, if $c>1/4$ we
have that the preimages by $f$ expand length in the vertical direction.
\begin{Lem} \label{c1}
There is $N(c)= N>0$ such that for $n \geq N$, every horizontal arc
$I,$ $l(I) = \Delta x$, every $b = b_1b_2 \cdot b_n$ it results
$l(Z_b(I)) > 4c \Delta x$ ($l$ is the euclidean length in the torus).
\end{Lem}
\begin{proof}
Given a segment $[x,x+\Delta x]\subset [0,1/2)$, the length
of its image by any of its branches: $Z_0,\, Z_1$  is given by
$$
\int_{x/2}^{x/2+\Delta x/2}\sqrt{1+\frac{c^2}{(1/2-s)^4}}\,ds\geq
\left.\frac{c}{1/2-s}\right|_{x/2}^{(x+\Delta x)/2}>
$$
$$
\left.\frac{c}{1/2-s}\right|_{0}^{\Delta x/2}=\frac{2c}{1-\Delta
x}-2c=2c(1+\Delta x+(\Delta x)^2+\cdots)-2c\geq 2c\Delta x\,.
$$

Analogously for the four second branches $Z_{00}, \, Z_{01},\,
Z_{10},\, Z_{11}$ we have that the graph of the pre-images is given
by the formula
$$g(u)=y_0 -\frac{2c}{1-2u}-\frac{2c}{1-4u}$$
in appropriate coordinates $(u,y)$, $u\in [h,h+\Delta x/4] \; \; h<1/4$.

Calculating the length of the graph we have
$$\int_{h}^{h+\Delta x/4}\sqrt{1+(g'(u))^2}\, du
\geq \int_{h}^{h+\Delta x/4}|(g'(u))|\, du=$$
$$\left. \frac{2c}{1-4u}+\frac{2c}{1-2u}\right|_h^{h+\Delta x/4}=$$
$$=\frac{2c}{1-h'}\left(\frac{1}{1-\frac{\Delta x}{1-h'}}-1\right)+
\frac{2c}{1-h'/2}\left(\frac{1}{1-\frac{\Delta
x/2}{1-h'/2}}-1\right)$$
$$\geq \frac{2c}{(1-h')^2}\Delta x+\frac{2c}{(1-h'/2)^2}\frac{\Delta
x}{2}\geq 2c\Delta x+c\Delta x=3c\Delta x\; $$
since $1\geq 1-h'>0$\; \; $(h'= 4h)$.

By induction we obtain in the general case ($n\geq 3$) that the length of
$Z_{{b}}([x_0,x_0+\Delta x],y)$ is bounded from above by
$$\ell(Z_{{b}}([x_0,x_0+\Delta x],y))\Delta x\geq
(2+1+\sum_{j=1}^{n-2}\frac{1}{2^j})c=(3+(1-\frac{1}{2^{n-2}}))\cdot
c \cdot \Delta x.$$
Thus,  the lemma follows
whenever  the length of the sequence ${b}$ is
greater or equal to $N=N(c)$.
\end{proof}

\begin{Lem}
\label{Lebesguepreserve}
Lebesgue measure $m$ is preserved by the map
$$f(x,y)=\left( 2x-\lfloor 2x\rfloor, \, y+ \frac{c}{|x-1/2|}-\left\lfloor y+
\frac{c}{|x-1/2|} \right\rfloor \right),\; c\in \R^+$$
\end{Lem}
\begin{proof}
  Given any
 small box $A=(a,b)\times(d,e)\subset (0,1)\times
 (0,1)$ it has two pre-images which are the subsets $A_0$ and $A_1$
 where $A_0$ is limited by the
 lines
 $$
 x=\frac{a}{2},\; x=\frac{b}{2},
 $$
 and the graph of the broken hyperbolas
 $$
 y=d-\frac{c}{1/2-x}-\left\lfloor
 d-\frac{c}{1/2-x}\right\rfloor,\;
 y=e-\frac{c}{1/2-x}-\left\lfloor e-\frac{c}{1/2-x}
 \right\rfloor,\; a/2\leq x\leq b/2;
 $$
 and $A_1$ is limited by the lines
 $$
 x=\frac{1+a}{2},\; x=\frac{1+b}{2},$$ and the graph of
the broken hyperbolas $$y=d-\frac{c}{x-1/2}-\left\lfloor
d-\frac{c}{x-1/2} \right\rfloor ,\; y=e-\frac{c}{x-1/2}
-\left\lfloor e-\frac{c}{x-1/2}\right\rfloor ,\;
(1+a)/2\leq x\leq (1+b)/2.
$$

 Calculating the area of $A_0$ by integration we obtain $(b-a)(e-d)/2$.
 Similarly for $A_1$. Summing both areas we obtain
 $(b-a)(e-d)=\mbox{Area}(A)$. Since the family of rectangles
 like $A$ gives a basis for the $\sigma$-algebra associated to
 Lebesgue measure $m$ we have proved that $m$ is $f$-invariant.
\end{proof}

 \begin{figure}[ht]
\begin{center}
\includegraphics[scale=0.38]{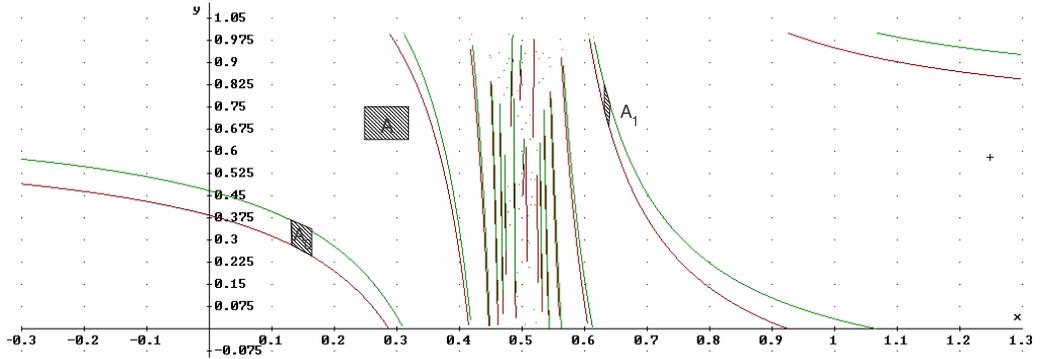}
\caption{A small rectangle $A$ and $A_0$, $A_1$, its pre-images. \label{fig.cien}}
\end{center}
\end{figure}

 Figure \ref{fig.cien} shows $A=[1/4,1/3]\times [2/3,3/4]$ and its
 pre-images, $A_0$ and $A_1$, where we have chosen
 $c=\pi-3\approx 0.1416$. The horizontal sides of $A$ are mapped
 into the broken graph of the hyperbolas, the top corresponding to
 the green line and the bottom to the red one.


\begin{Teo}
 \label{teoMixingMedida}
Let $c>1/4$. Then the map $f$ is mixing with respect to Lebesgue
measure.
\end{Teo}
\begin{proof}
There is no loss of generality choosing
$A$ and $R$ as rectangles contained in $Q$, since the family of
rectangles generates the $\sigma$-algebra associated to the Lebesgue
measure \cite{Wa}, Theorem 1.17. We have to show
$$
\lim_{n\to\infty}|m(f^{-n}(A)\cap R) - m(A)\cdot m(R)|=0.
$$
For this we proceed as follows. Let $W=(x^R,y^R)\in R$, be the
center point of $R$ and $Z=(x^A,y^A)$ the center point of $A$.
We use the binary decomposition of  $x^R\sim b_1b_2\cdots $
and let $\hat{x}\sim \overline{b_1b_2\cdots b_N}= b,$ that is,
$\hat{x}$ is the $N$-periodic point of the map $x \mapsto 2x$ mod $1$.

Taking $N$ sufficiently large, we have that the vertical line
$x=\hat{x}$ crosses the rectangle $R$ nearby its central point $W$.
Indeed, $|\hat{x}-x^R|<2^{-N}$.
If $x_{{b}}=\Pi_x(Z_{{b}})$, as a consequence of the remarks
after formula (\ref{badalo})  we have
$|x_{b}-\hat{x}|<2^{-N}$, and also $|x_{{b}}-x^R|<2^{-N}$.

 Moreover, for $n\geq N$, with $N$ large enough, the $n$-th
 pre-image of any horizontal segment of length $\Delta x$ in
$A$, is almost vertical and their length is greater than $2^{\lfloor
n/N\rfloor} c \Delta x$, see Lemma \ref{c1}.

\begin{claim} \label{c2}
The distance $h_a$ between the pre-images of the top and the bottom
of $A$ by $Z_{{a}}$, where $a=a_1a_2\cdots a_n$, see
(\ref{badalo}), is
\begin{equation} \label{eq2}
h_a\approx \frac{m(A)}{2^n\cdot \Delta x_A\cdot }L_a,
\end{equation}
where $\Delta x_A$ is the length of the bottom of $A$ and  $L_a > (4c)^{[n/N]}$.
\end{claim}
\begin{proof}
$Z_{{a}}(A)$ has measure equal to $m(A)/2^n$. Moreover, as a consequence
of Lemma \ref{c1} the length $l_a$ of $\Pi_x(Z_{{a}}(A))$ is given by $\Delta x_A\cdot
L_a$. Thus the height $h_a$ of the almost
parallelogram given by $Z_{{a}}(A)$ is
$$ h_a\approx \frac{m(A)}{2^n\cdot l_a}=\frac{m(A)}{2^n\cdot \Delta
x_A\cdot L_a}
$$ proving the claim.
\end{proof}
Returning to the proof of Theorem \ref{teoMixingMedida} note that if
$n$ is large enough, as a consequence of the remarks
after equations (a) and (b) at (\ref{badalo}), $Z_{{a}}(A)$ is a long ``vertical"
strip almost uniformly distributed in the torus. Then there are
$ \Delta x_A L_a \Delta x_R$ strips cutting $R$.
By claim \ref{c2} each
turn leaves a strip in $R$ of area
$$\frac{\Delta y_R \cdot m(A)}{\Delta x_A\cdot 2^n\cdot
L_a}.$$ Then the total area equals
$$\frac{\Delta y_R \cdot m(A)\cdot \Delta x_A\cdot
L_a \cdot \Delta x_R}{\Delta x_A\cdot 2^n\cdot
L_a}=\frac{\Delta y_R \cdot m(A) \cdot \Delta x_R}{ 2^n}.
$$
Since the number of pre-images satisfying the previous computations
is $2^n$ we obtain that
$$
m(f^{-n}(A)\cap R)\approx\frac{M(R) \cdot m(A)}{ 2^n} 
\cdot 2^n\to m(A)\cdot m(R)\quad \mbox{for }\quad n\to\infty\,,
$$
finishing the proof.
\end{proof}

Lemma \ref{Lebesguepreserve} together with Theorem \ref{teoMixingMedida} prove Theorem \ref{Teo2}.


\section{Rate of mixing}
Next we prove that the rate of mixing for $c>1/4$ is exponential.
To do so we start with an auxiliary lemma.

\begin{Lem} \label{extremos}
Given $b\in [0,1)$, $$b=\frac{b_1}{2}+\frac{b_2}{2^2}+\cdots +
\frac{b_{N-1}}{2^{N-1}}+\frac{b_N}{2^N} \sim b_1b_2\cdots
b_{N-1}b_N, $$ there is at most one point in $(0,1)$ where the
derivative $y_b'(x_b)$ can vanish. Moreover, if $y_b'(x_b)$
vanishes, the value of $x$ at which $y'_b=0$ is between
$\Sigma_{i=1}^j b_i/2^i$ and $\Sigma_{i=1}^j b_i/2^i + 1/2^N$ where
$b_j$ is the first digit in $b_1b_2\cdots b_{N-1}b_N $
different from $b_N$ (i.e.: $b_N=b_{N-1}=\cdots =b_{j+1}\neq b_j$).
\end{Lem}
\begin{proof}
We will repeatedly use that $y_b$ and $x_b$ are given by equations (a)
and (b) at (\ref{badalo}). Let $(x_0,y_0)$ be given.

For $N=1$ we have that
$$y_{b_1}=y_0-\frac{c}{|x_{b_1}-1/2|}\mod (1)=y_0-\frac{c\cdot
(2b_1-1)}{x_{b_1}-1/2}\mod (1)$$ where $x_{b_1}=\frac{x_0+b_1}{2}\in
(\frac{b_1}{2},\frac{b_1+1}{2})$. 
Observe that $2b_1-1=-1$ if $b_1=0$ and $2b_1-1=1$ if $b_1=1$. Thus
$$y_{b_1}'(x_{b_1})=\frac{c\cdot (2b_1-1)}{(x_{b_1}-1/2)^2}\quad
\mbox{which does not vanish whenever it exists.}$$

For $N=2$, on account that
$$
x_{b_1b_2}=\frac{x_{b_2}+b_1}{2}\in
(\frac{b_1}{2}+\frac{b_2}{2^2},\frac{b_1}{2}+\frac{b_2}{2^2}+ \frac{1}{2^2})
$$
we have $x_{b_2}=2x_{b_1b_2}-b_1$ and
$$
y_{b_1b_2}(x_{b_1b_2})=y_{b_2}-\frac{c}{|x_{b_1b_2}-1/2|}=y_0-\frac{c}{|x_{b_2}-1/2|}-
 \frac{c}{|x_{b_1b_2}-1/2|} \mod(1)=$$
 $$=y_0-\frac{c\cdot
 (2b_2-1)}{(2x_{b_1b_2}-b_1-1/2)}-\frac{c\cdot
 (2b_1-1)}{(x_{b_1b_2}-1/2)}\mod (1)$$
from which we conclude that
$$y_{b_1b_2}'(x_{b_1b_2})=\frac{2c\cdot
 (2b_2-1)}{(2x_{b_1b_2}-b_1-1/2)^2}+\frac{c\cdot
 (2b_1-1)}{(x_{b_1b_2}-1/2)^2}\,,$$
which does not change sign if $b_1=b_2$ or changes sign only once in
its domain if $b_1\neq b_2$.

In general the expression of $y_b=y_{b_1b_2\cdots b_N}$ as a
function of $x_b=x_{b_1b_2\cdots b_N}$ is given by
$$y_b(x_b)=y_0-\frac{c\cdot(2b_N-1)}{(2^{N-1}x_b-(2^{N-2}b_1+2^{N-3}b_2+\cdots
+2b_{N-2}+b_{N-1})-1/2)}\, -$$
$$\frac{c\cdot(2b_{N-1}-1)}{(2^{N-2}x_b-(2_{N-3}b_1+2^{N-4}b_2+\cdots
+2b_{N-3}+b_{N-2})-1/2)}-\cdots
-\frac{c\cdot(2b_1-1)}{(x_b-1/2)}\mod(1)$$ from which the derivative
of $y_b$ with respect to $x_b$, whenever it exists, is given by
$$y_b'(x_b)=\frac{2^{N-1}c\cdot(2b_N-1)}{(2^{N-1}x_b-(2_{N-2}b_1+2^{N-3}b_2+\cdots
+2b_{N-2}+b_{N-1})-1/2)^2}\, +$$
$$\frac{2^{N-2}c\cdot(2b_{N-1}-1)}{(2^{N-2}x_b-(2_{N-3}b_1+2^{N-4}b_2+\cdots
+2b_{N-3}+b_{N-2})-1/2)^2}+\cdots
+\frac{c\cdot(2b_1-1)}{(x_b-1/2)^2}\, .$$ The last expression can be written as $$y_b'(x_b)=\frac{c\cdot(2b_N-1)}{2^{N-1}
\left(x_b-(\frac{b_1}{2}+\frac{b_2}{2^2}+\cdots
+ \frac{b_{N-2}}{2^{N-2}}+\frac{b_{N-1}}{2^{N-1}})-\frac{1}{2^N}\right)^2}\,
+$$
$$\frac{c\cdot(2b_{N-1}-1)}{2^{N-2}\left(x_b-(\frac{b_1}{2}+\frac{b_2}{2^2}+\cdots
+\frac{b_{N-3}}{2^{N-3}}+\frac{b_{N-2}}{2^{N-2}})-\frac{1}{2^{N-1}}\right)^2}+\cdots
+\frac{c\cdot(2b_1-1)}{\left(x_b-\frac{1}{2}\right)^2}\, .$$
$$\mbox{Hence }\quad y_b'(x_b)=\frac{c\cdot(2b_N-1)}{2^{N-1}\left(x_b- 0.b_1b_2\cdots b_{N-2}b_{N-1}-
\frac{1}{2^N}\right)^2}\,
+$$
$$\frac{c\cdot(2b_{N-1}-1)}{2^{N-2}\left(x_b-0.b_1b_2\cdots b_{N-3}b_{N-2}-
\frac{1}{2^{N-1}}\right)^2}+\cdots
+\frac{c\cdot(2b_1-1)}{\left(x_b-\frac{1}{2}\right)^2}\,$$ where we
have written $\frac{b_1}{2}+\frac{b_2}{2^2}+\cdots
+\frac{b_{N-2}}{2^{N-2}}+\frac{b_{N-1}}{2^{N-1}}=0.b_1b_2\cdots
b_{N-1}$, and similarly for the other terms. Observe that each term in
the expression or $y'_{b}(x_b)$ is positive or negative depending on the
values of $b_i$ and that there are
$N$ vertical asymptotes for (the lift to $\R^2$ of) $y_b(x_b)$ and
for $y'_b(x_b)$ which are located in
\begin{equation} \label{eq.asymp}
x=0.b_1b_2\cdots b_{N-1}+\frac{1}{2^N}\,,\quad x=0.b_1b_2\cdots b_{N-2}+\frac{1}{2^{N-1}}\,,\;
 \end{equation}
 $$x=0.b_1b_2\cdots b_{N-3}+\frac{1}{2^{N-2}}\,,\quad
\ldots\,,\quad  x=0.b_1+\frac{1}{2^2}\,,\quad x=\frac{1}{2}\, .$$

\begin{claim} \label{unicozero}
All the terms with positive sign in $y_b'(x _b)$ have their asymptotes
at points of coordinates less than those which have negative sign.
Moreover $y'_b(x_b)$ will vanish only once for an $x_\xi$ located
between the closest asymptotes of different sign.
\end{claim}
\begin{proof}
Let us prove the claim by induction.

For $N=1$ there is nothing to prove. For $N=2$, if $b_2=0$ and
$b_1=1$ then we have an asymptote $x=\frac{1}{2}$ and the other
$x=\frac{3}{4}$ and the derivative is
$$y_{b_1b_2}'(x_{b_1b_2})=y_{10}'(x_{10})=
\frac{-c}{2(x_{10}-3/4)^2}+\frac{c}{(x_{10}-1/2)^2}\,,$$
If $b_2=1$ and $b_1=0$ then we have an asymptote $x=1/4$ and the
other $x=1/2$ and the derivative is
$$y_{b_1b_2}'(x_{b_1b_2})=y_{01}'(x_{01})=
\frac{c}{2(x_{01}-1/4)^2}+\frac{-c}{(x_{01}-1/2)^2}\,.$$ Hence the
claim is true for $N=1,\,2$.

Assume that the claim is true for $b_2b_3\cdots b_N$ and let us
prove it for $b=b_1b_2b_3\cdots b_N$. If $b_1=0$ then all the values
of the asymptotes are divided by 2 and the corresponding asymptotes
of the positive terms in $y'_b(x_b)$ rest to the left of the smallest
asymptote corresponding to a negative term (if there is any one). By
induction the difference between the smallest asymptote of  a
negative term and the largest asymptote of a positive term is
$1/2^{N-1}$ for $b_2b_3\cdots b_N$, when we divide by 2 the
difference becomes $1/2^N$. Moreover, all terms are less than $1/2$
and we add a negative term corresponding to the asymptote $x=1/2$.
Thus the claim is true for $b_1=0$.

If $b_1=1$ then all the values of the asymptotes are divided by 2
and to these values we add $1/2$. Therefore the corresponding
asymptotes of the positive terms in $y'_b(x_b)$ rest to the left of
the smallest asymptote corresponding to a negative term (if there is
any). The difference between the smallest asymptote of  a negative
term and the largest asymptote of a positive term becomes $1/2^{N}$
as above. All terms are greater than $1/2$ and we add a positive
term corresponding to the asymptote $x=1/2$.

If it were the case that $b_2=b_3=\cdots b_N=0$ but $b_1=1$, then
all terms should be negative till the last one. After the final step
a positive term appears with asymptote $x=1/2$ while the leftmost
negative term will be $1/2+1/2^{N}$. Similarly if $b_2=b_3=\cdots
b_N=1$ but $b_1=0$, then all terms should be positive till the last
one. After the final step a negative term appears with asymptote
$x=1/2$ while the rightmost positive term will be $1/2-1/2^{N}$.
 Now the proof of the claim is complete.
\end{proof}

 The proof of the lemma follows readily from Claim \ref{unicozero},
 doing computations similar to those for the case $N=2$.

\end{proof}

\begin{Obs}
Although the number of asymptotes is $N$, since $x\in [0,1]$
we have that $x_b$ belongs to an interval of length $\frac{1}{2^N}$
and in the general case only two of the asymptotes fall in the domain of $x_b$.
The distance between these asymptotes is $\frac{1}{2^N}$.
\end{Obs}
\begin{Obs}
Let $A=[x_A-\Delta x/2,x_A+\Delta x/2]\times [y_A,y_A+ \Delta y]$.
From Lemma \ref{c1} it follows immediately that there is a first
$N_0=N_0(\Delta x)$ such that $\ell(Z_b([x_A-\Delta
x/2,x_A+\Delta x/2]))\geq 1$ if $|b|=N_0$.
\label{rN}
\end{Obs}

The following lemma says that far from the asymptotes the growth of lengths
of the pre-images is bounded from above.
\begin{Lem} \label{creceacotado}
Let $c>0$.
Given $K>0$ and $\epsilon>0$ there is $\delta>0$ such that if
$0<\Delta x\leq \delta$ then $N_0>K$ for a subset of $\Sigma$ of
measure greater or equal than $1-K\epsilon$.
\end{Lem}
\begin{proof}
Let us choose  a vertical strip $S_\epsilon=
[1/2-\epsilon/2,1/2+\epsilon/2]\times [0,1)$ and assume that
$I=[x-\Delta x /2,x+\Delta x /2]\times
\{y\}$ does not intersect $S_\epsilon$. Let us bound from above the
length of the pre-images of $I$.
Recall that these pre-images are given by

$Z_0=\left(\frac{x}{2}, y - \frac{2c}{1-x}-\lfloor
y-\frac{2c}{1-x}\rfloor\right)=(x_0,y_0)$ and

 $Z_1=\left(\frac{1+x}{2},
y-\frac{2c}{x}-\lfloor y- \frac{2c}{x}\rfloor\right)=(x_1,y_1).$

Let us assume that $x+\Delta x/2<1/2-\epsilon/2$, the other cases are similar.
This implies in particular that $1-x>\epsilon$. We obtain:
$$\ell (Z_0(
I))=
\int_{x/2-\Delta x/2}^{x/2+\Delta x/2}\sqrt{1+\frac{c^2}{(1/2-s)^4}}ds
=$$
$$
\int_{x/2-\Delta x/2}^{x/2+\Delta x/2}\left(\frac{(1/2-s)^4+c^2}{(1/2-s)^4}\right)^{1/2}ds\leq
\int_{x/2-\Delta x/2}^{x/2+\Delta x/2}\frac{(1/16+c^2)^{1/2}}{(1/2-s)^2}ds\leq
$$
$$\left.\frac{c+1/4}{1/2-s}\right|_{(x-\Delta x)/2}^{(x+\Delta x)/2}=\frac{2c+1/2}{1-x}
\left[\frac{1}{1-\Delta x/(1-x)}-\frac{1}{1+\Delta x/(1-x)}\right]=$$
$$\frac{2c+1/2}{(1-x)^2}2\Delta x \left[ 1+\frac{(\Delta x)^2}{(1-x)^2}+
\frac{(\Delta x)^4}{(1-x)^4}+\cdots \right]< \frac{2c+1/2}{\epsilon^2}
\left(\frac{1}{(1-\delta^2/\epsilon^2)}\right) 2\delta$$
where we have put $\Delta x<\delta<\epsilon$.
This gives the upper bound for the length of a pre-image given by
$$\ell (Z_0(
I))<\frac{4c+1}{\epsilon^2-\delta^2}\delta$$
The same bound is valid for the case $x>1/2$ and for the other pre-image given by $Z_1$.

Let us denote by $b^{(K)}=b_1b_2b_3\cdots b_K$ the finite subsequence given by the first $K$ terms of $b=\{b_n\}_{n\in\N}\in\Sigma$  and
$$X_K(x,y)=\{(x_{b^{(K)}},y_{b^{(K)}})\,:\, f^K(x_{b^{(K)}},y_{b^{(K)}})=(x,y)\}\,.$$
There is $\epsilon>0$ such that
if $f^j([u-\Delta x/2,u+\Delta x/2
],v)\cap [1/2-\epsilon/2, 1/2+\epsilon/2]\times [0,1)=\emptyset$ for all $j=0,1,2,\ldots ,K-1$
then the length of $Z_{b^{(K)}}(([u-\Delta x/2,u+\Delta x/2],v))<\left(\frac{4c+1}{\epsilon^2-\delta^2}\delta\right)^K$ from which the thesis follows choosing $\delta$ small enough.

\end{proof}

By Remark \ref{rN} after $N_0$ iterations the length of $Z_{b^{(N_0)}}$ is at least 1.
  Thus if $k_0$ denotes the time needed for a monotone arc $\gamma$
  to duplicate its length when computing $Z_{(b)^{(k_0)}}$ (see Lemma \ref{c1})
 we obtain the following
\begin{Cor} \label{fatias}
If $N=N_0+k_0h$, $h\geq 0$, then for $b$ such that $|b|=N$ we have that
$$\ell(Z_b([x_A-\Delta x/2,x_A+\Delta x/2]\times\{y\}))\geq 2^h.$$
\end{Cor}

Corollary \ref{fatias} implies that the pre-image
$Z_b(A)$ has $2^h$ connected components in $[0,1]\times [0,1]$ which are almost
vertical strips. The value of $N_0$ is bounded but depends on the length of $\Delta x$
and the position of $x_A$.

 The next lemma estimates the width of each
 of these strips.
Before we state it,
let us sort out the intersections between $Z_b(A)
$ and $[0,1]\times [0,1]$ in the following way:
\medbreak

\noindent $(\star)$\/ The image in $\R^2$ of
the top side $T=[x_A-\Delta x/2,x_A+\Delta x/2]\times \{y+\Delta y\}$
of $A$ is an arc almost parallel to the vertical axis $Oy$ with
reverse orientation. We assign the label $n$ to the connected component
of this arc whose projection covers the interval $[-n+1,-n]$ in $Oy$
(see Figures \ref{fig.102} and \ref{fig.103}).
Similarly for the bottom segment $B=[x_A-\Delta x/2,x_A+\Delta x/2]\times \{y\}$.

 \begin{figure}[ht]
\begin{center}
\includegraphics[scale=0.5]{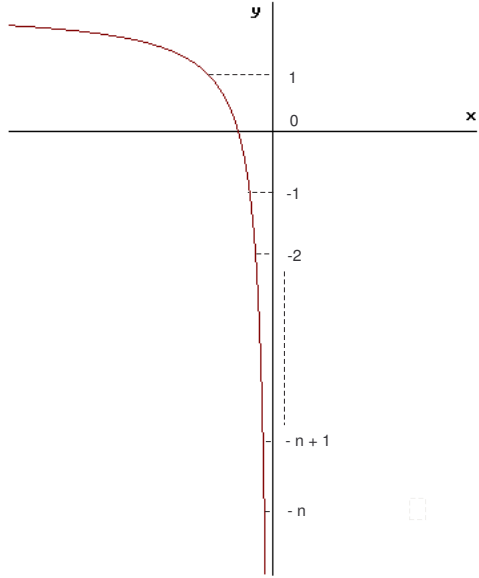}
\caption{The image of $T$ in $\R^2$. \label{fig.102}}
\end{center}
\end{figure}

 \begin{figure}[ht]
\hspace{1.5cm}\includegraphics[scale=0.42]{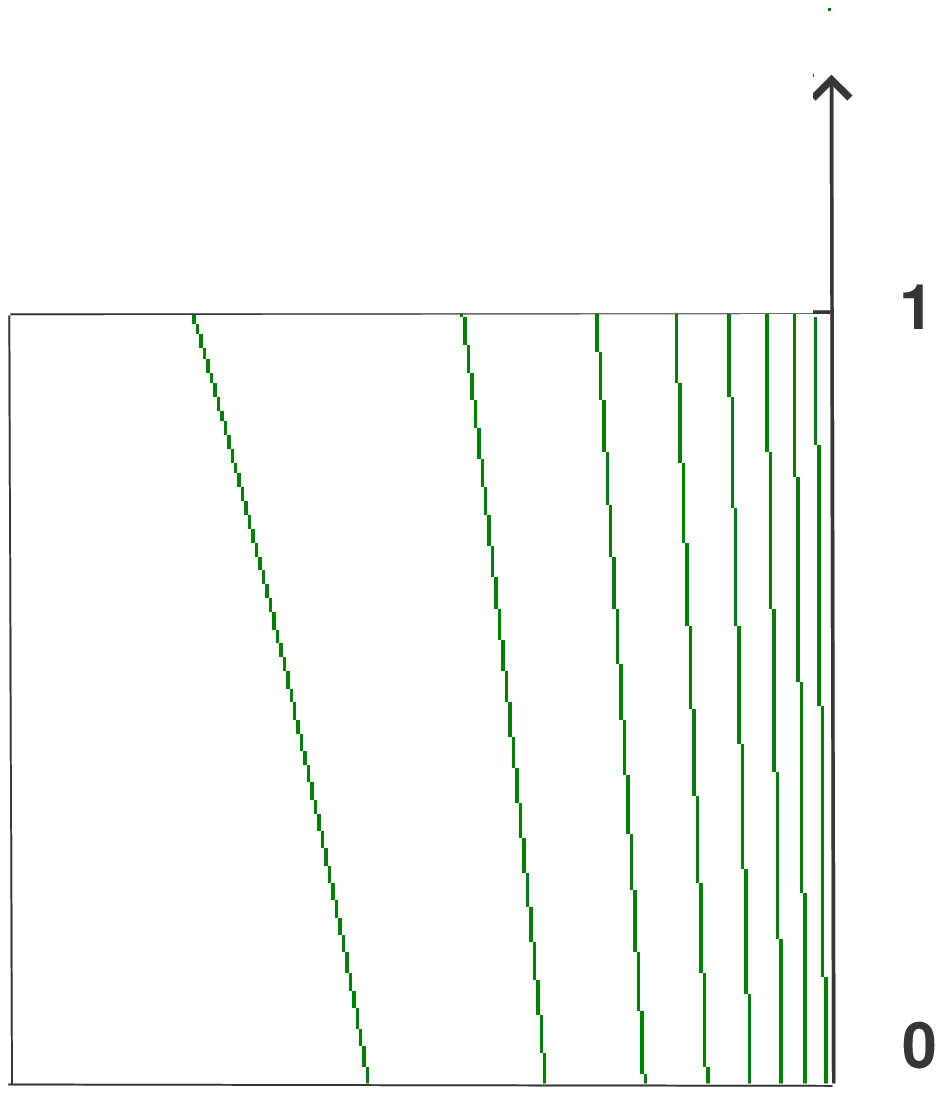}
\caption{The image of $T$ in $[0,1]\times [0,1]$. \label{fig.103}}
\end{figure}

\begin{Lem} \label{fajas}
Let  $A=[x_A-\Delta x/2,x_A+\Delta x/2]\times [y_A,y_A+\Delta y]$ and
$N=N_0+k_0\,h$, $h\geq 0$, be  as above. Denote by $T$ the top and $B$
the bottom sides of the rectangle $A$. Then, for $b$ such that $|b|=N$
there is a constant $C>0$ such that
$$\dist(Z_b(T)_n,Z_b(B)_n)\leq C\frac{\Delta y}{2^N n^2}\, ,$$ where $Z_b(T)_n$,
and $Z_b(B)_n$ are the $n^{th}$-connected component of $Z_b(T)$ and $Z_b(B)$ respectively.
\end{Lem}
\begin{proof}
For the bottom side $B$ the expression of $y_b=y_{b_1b_2\cdots b_N}$
in $\R^2$ as a function of $x_b=x_{b_1b_2\cdots b_N}$ is given by
(see the proof of Lemma \ref{extremos})
$$y_b(x_b)=y_A-\frac{c\cdot(2b_N-1)}{(2^{N-1}x_b-(2^{N-2}b_1+2^{N-3}b_2+\cdots
+2b_{N-2}+b_{N-1})-1/2)}\, -$$
$$\frac{c\cdot(2b^{N-1}-1)}{(2^{N-2}x_b-(2{N-3}b_1+2^{N-4}b_2+\cdots
+2b_{N-3}+b_{N-2})-1/2)}-\cdots
-\frac{c\cdot(2b_1-1)}{(x_b-1/2)}\, .$$
This curve has $N$ asymptotes
$x_1,x_2,\ldots,x_N$,  see equation (\ref{eq.asymp}).
Close to one of each asymptotes, say $x_b=x_j$,
$y_b(x_b)$ can be written as
$$y_b(x_b)=y_A+H(x_b)+\frac{c\,(2b_j-1)/2^{j-1}}{(x_b-x_j)}\,,\quad $$
where $H(x_b)$ has a finite limit $H_j$ when $x_b\to x_j$.
Similarly for the top side $T$ we have
$$y_b(x_b)=y_A+\Delta y+H(x_b)+\frac{c\,(2b_j-1)/2^{j-1}}{(x_b-x_j)}\,.\quad $$
The only values that give asymptotes in the domain of $x_b$
correspond to the case $j=N$ that give
$$y_b(x_b)=y_A+H(x_b)+\frac{c\,(2b_N-1)/2^{N-1}}{(x_b-x_N)}\,,\quad $$
and
$$y_b(x_b)=y_A+\Delta y+H(x_b)+\frac{c\,(2b_N-1)/2^{N-1}}{(x_b-x_N)}\,.\quad $$
For a fixed $y_b$, varying from $-n+1$ to $-n$, we have
$$ (\star_1) \,\,y_b= y_A+H(\tilde x_b)+\frac{c\,(2b_N-1)/2^{N-1}}{(\tilde{x_b}-x_N)}
\quad \mbox{and}\quad(\star_2)\,\,y_b=y_A+\Delta y+H(\widehat{x_b}) +\frac{c\,(2b_N-1)/2^{N-1}}{(\widehat{x_b}-x_N)}\,.$$
For a given $\varepsilon>0$, there is $n_0$ such that for $n>n_0$ it
holds that $|H(x_b)- H_N|<\varepsilon$. Thus, from the first equation
$(\star_1)$ we have that, for $n>n_0$, it holds
$$y_b \approx y_A+H_N+\frac{c\,(2b_N-1)/2^{N-1}}{(\tilde x_b-x_N)}\Longrightarrow
\tilde x_b\approx x_N+\frac{c(2b_N-1)}{(y_b-y_A-H_N)2^{N-1}}\,.$$
From the second one $(\star_2)$
we obtain that $$y_b \approx y_A+H_N+\Delta y+\frac{c\,(2b_N-1)/2^{N-1}}
{(\widehat x_b-x_N)}\Longrightarrow \widehat x_b\approx x_N+\frac{c(2b_N-1)}
{(y_b-y_A-H_N-\Delta y)2^{N-1}}\, .$$
 This implies that
 $$|\widehat x_b- \tilde x_b|\approx\left|\frac{c(2b_N-1)}{(y_b-y_A-H_N-\Delta y)2^{N-1}} -\frac{c(2b_N-1)}{(y_b-y_A-H_N)2^{N-1}}\right|\, . $$
 Taking into account that $-n+1>y_b>-n$ and $2b_N-1=\pm1$ we conclude that
 $$|\widehat x_b- \tilde x_b|\approx \frac{c}{2^{N-1}}\left(\frac{1}{(-n-y_A-H_N
 -\Delta y)}-\frac{1}{(-n-y_A-H_N)}\right)= $$
$$ \frac{c}{2^{N-1}}\frac{\Delta y}{(n+y_A+H_N)(n+y_A+H_N+\Delta y)}<
\frac{c}{2^{N-1}}\frac{\Delta y}{n^2}< C\frac{\Delta y}{2^N n^2}\, .$$
 Here we have chosen the constant $C>0$ such that the inequality
 holds for all $ n\geq 1$ and not only for $n>n_0$.
\end{proof}

\begin{Teo} \label{decay}
There is $0<\theta<1$ such that after $N=N_0+k_0\,h$ iterates by any
branch $Z_b$ of $f^{-1}$, the Lebesgue measure of the set of points
that has returned to $A$ is greater or equal to $m(A)\cdot
\theta$.
\end{Teo}
\begin{proof}
Let $\ell(Z_b([x_A-\Delta x/2,x_A+\Delta x/2]\times\{y\}))=2^h$ with
$|b|=N$ (Corollary \ref{fatias}), and assume that $\Pi_x(Z_b(x_a,y)\in [x_A-\Delta
x/2+\frac{1}{2^N},x_A+\Delta x/2-\frac{1}{2^N}]$. Then $Z_b(A)$ cuts
$A$ in at least
$2^h$ strips which are almost vertical except perhaps for one
which becomes from that strip where the derivative $y_b'(x_b)$ can
vanish. We don't take into account this strip so that we either have
$2^h$ almost vertical strips or $(2^h-1)$ of them. Since $2^h$ is a
lower bound we will consider that the number of strips is $2^h$ anyway.
By Lemma \ref{fajas} the (almost) vertical
sides of the strips which are at a distance between them
$\approx \frac{C\Delta y}{2^N n^2}$ intersected with $A$ are mapped by $f^N$, $N=|b|$,
in part of the horizontal sides of length  proportional to $ C\Delta y/n^2$ with
$N_0\leq n\leq 2^h$. Thus the area covered by the $f^N$-image of one of
the strips is about a constant $D$ multiplied by the length $\Delta y/n^2$
of the horizontal sub-intervals, by the height $\Delta y$ which gives
$$\mbox{Area}_n\approx D\cdot \frac{\Delta y}{n^2}\cdot \Delta y\, .$$
It follows that the area of the $f^N$-image of the $2^h$
strips is
$$\sum_{n=N_0}^{2^h }\mbox{Area}_n=D\cdot(\Delta
y)^2\sum_{n=N_0}^{2^h}\frac{1}{n^2}\, .$$
Since any point in $A$ has
 $2^N$ preimages by the different $Z_b$, $|b|=N$,
 we have to divide
this number by $2^N $ in order not to multiple count. This gives us
$$D\cdot(\Delta y)^2\frac{\sum_{n=N_0}^{2^h
}\frac{1}{n^2}}{2^N }\,.$$
Since the number of preimages from $N_0$ to $N$ that cut $A$ is
given by the action of $x\mapsto 2x\mod (1)$ in $[0,1)$, which is
Bernoulli, we have that this number is $\approx 2^{k_0h}\Delta x$.
Hence we have that the area of the set of points that have returned
after $N$ preimages is
$$\mbox{covered area}\approx D\cdot(\Delta
y)^2\frac{\sum_{n=N_0}^{2^h}\frac{1}{n^2}}{2^N}\Delta x \cdot
2^{k_0h}=\Delta x\Delta y \left(D\cdot\Delta
y\cdot\frac{2^{k_0h}}{2^{N_0+k_0h}}\cdot
\sum_{n=N_0}^{2^h}\frac{1}{n^2}\right)=
$$
$$=m(A)\left(D\cdot\Delta
y\cdot 2^{-N_0}\cdot
\sum_{n=N_0}^{2^h}\frac{1}{n^2}\right)=m(A)\cdot \theta,\,$$
$$\mbox{where}\quad \theta=D\cdot\Delta y\cdot 2^{-N_0}\cdot
\sum_{n=N_0}^{2^h}\frac{1}{n^2}<1\,.$$

Therefore the measure of the set of points that have not returned
yet is at most $m(A)(1-\theta)$.
After taking $2^N$ new preimages
(i.e.: by backward iteration $N$ times following all the possible
$2^N$ branches $Z_b$, $|b|=N$, from the new starting point) we cover
$m(A)(1-\theta)\theta$ which implies that it rests at most
$m(A)(1-\theta)^2$ points that have not returned to $A$ yet. We
conclude by induction that for $|b|=nN$ the measure of points not
covered after taking all $2^{nN}$ pre-images is less than
$m(A)(1-\theta)^n\to 0$ when $n\to\infty$.
\end{proof}

\begin{Cor}
We have that after $n\,N$ iterates, the Lebesgue measure of
points that have not yet returned is less than $m(A)(1-\theta)^n$.
\end{Cor}

 The next corollary gives that the rate of recurrence of $f$ is exponential.

\begin{Cor} \label{4.7}
 It holds that
$\lim_{n\to\infty} |m(f^{-n}(A)\cap A)-m^2(A)|=0$ exponentially fast.
\end{Cor}
\begin{proof}
Note that by Theorem \ref{decay} we have that the measure of points that
have returned to $A$ after $N$ iterations is
$$
m(A)\left(D\cdot\Delta y\cdot 2^{-N_0}\cdot\sum_{n=N_0}^{2^h}\frac{1}{n^2}\right)\, .
$$
We may write this expression as
$$m(A)\left(D\cdot\Delta x\cdot\Delta y\cdot \frac{2^{-N_0}}{\Delta x}
\cdot\sum_{n=N_0}^{2^h\Delta x}\frac{1}{n^2}\right)=(m(A))^2\left(D\cdot
\frac{2^{-N_0}}{\Delta x}\cdot\sum_{n=N_0}^{2^h\Delta x}\frac{1}{n^2}\right)\,.$$
 For $N_0$ sufficiently large  we have that $\lambda=\left(D\cdot
 \frac{2^{-N_0}}{\Delta x}\cdot\sum_{n=N_0}^{n=2^h}\frac{1}{n^2}\right)<1$
and therefore we obtain that after $n=hN$ iterations
 $$ |m(f^{-n}(A)\cap A)-m^2(A)|\leq |(m(A))^2\big(1-\lambda^{\left[\frac{n;}{N}
 \right]}\big)-(m(A))^2|=(m(A))^2\lambda^{\left[\frac{n}{N}\right]}$$
 Putting $\lambda^{1/N}=\tau<1$ we have
 $$|m(f^{-n}(A)\cap A)-m^2(A)|\leq (m(A))^2\tau^n$$
 proving the thesis.
\end{proof}

The following theorem is similar to Theorem \ref{decay}.
\begin{Teo} \label{decay2}
Given small rectangles
$A=[x_A-\frac{\Delta x_A}{2},x_A+\frac{\Delta x_A}{2}]\times [y_A,y_A+\Delta y_A]$ and
$B=[x_B-\frac{\Delta x_B}{2},x_B+\frac{\Delta x_B}{2}]\times [y_B,y_B+\Delta y_B]$
there is $0<\theta<1$ such
that the set of points of $A$ that has visited $B$ after $N$ iterates
is greater or equal than $m(B)\cdot \theta$.
\end{Teo}
\begin{proof}
By Corollary \ref{fatias} we have that
$\ell(Z_b([x_A-\frac{\Delta x_A}{2},x_A+\frac{\Delta x_A}{2}]\times\{y\}))=2^h$ with
$|b|=N$.
Assume that $\Pi_x(Z_b(x_a,y))\in [x_B-\frac{\Delta
x_B}{2}+\frac{1}{2^N},x_B+\frac{\Delta x_B}{2}-\frac{1}{2^N}]$.
Then $Z_b(A)$ cuts
$B$ in $2^h$ strips which are almost vertical except perhaps for one of them
corresponding to that strip where the derivative $y_b'(x_b)$
vanishes. We don't take into account this strip so that we either have
$2^h\Delta x_A$ or $(2^h-1)\Delta x_A$ almost vertical strips. The area of
$Z_b(A)$ is $m(A)/2^N$.

By Lemma \ref{fajas} and taking into account the sorting given at $(\star)$,
the intersection of the (almost) vertical
sides of the strips with $B$ are mapped by $f^N$, $N=|b|$,
in a subsegment of the horizontal sides of $A$ with length  $\approx C\Delta y_B/n^2$,
$N_0\leq n\leq 2^h$, recall Corollary \ref{fatias}.

\noindent Thus, the area covered by the $f^N$-image of
the  $n^{th}$-strip is given by
$$(\mbox{Area in }A)_n\approx D\cdot \frac{\Delta y_B}{n^2}\cdot \Delta y_A\, ,$$
where  $D$ is a constant,   $\Delta y_B$ is the length of the vertical side of $B$,
 and $\Delta y_A$ is the length
of the vertical side of $A$.

Therefore, the area of the $f^N$-image of all the $(2^h-N_0)$ strips is
$$\sum_{n=N_0}^{2^h }\mbox{Area}_n=D\cdot(\Delta
y_A)(\Delta y_B)\sum_{n=N_0}^{2^h}\frac{1}{n^2}\, .$$
Since any point in $A$ has
 $2^N$ pre-images by the different $Z_b$, $|b|=N$, we have to divide
this number by $2^N$ in order not to multiple count. This gives us
$$D\cdot(\Delta y_A)(\Delta y_B)\frac{\sum_{n=N_0}^{2^h}\frac{1}{n^2}}{2^N}\,.$$
Since the number of pre-images from $N_0$ to $N=N_0+2^{k_0h}$ that cut $B$ is
given by the action of $x\mapsto 2x\mod (1)$ in $[0,1)$, which is
Bernoulli, we have that this number is $\approx 2^{k_0h}\Delta x_B$.
Hence we have that the area of the set of points that have cut $B$
after $N$ pre-images is
$$\approx D\cdot(\Delta
y_A)(\Delta y_B)\frac{\sum_{n=N_0}^{2^h\Delta x_A}\frac{1}{n^2}}{2^N}\Delta x_B \cdot
2^{k_0h}=\Delta x_B\Delta y_B \left(D\cdot\Delta y_A \cdot\frac{2^{k_0h}}{2^{N_0+k_0h}}\cdot
\sum_{n=N_0}^{2^h}\frac{1}{n^2}\right)=
$$
$$=m(B)\left(D\cdot\Delta
y_A\cdot 2^{-N_0}\cdot
\sum_{n=N_0}^{2^h}\frac{1}{n^2}\right)=m(B)\cdot \theta,\,$$
$$\mbox{where}\quad \theta=D\cdot\Delta y_A\cdot 2^{-N_0}\cdot
\sum_{n=N_0}^{2^h\Delta x}\frac{1}{n^2}<1\,.$$

Therefore the measure of the set of points of $A$ that have not visited
yet the set $B$ is at most $m(B)(1-\theta)$. After taking $2^N$ new pre-images
(i.e.: by backward iteration $N$ times following all the possible
$2^N$ branches $Z_b$, $|b|=N$, from the new starting point) we cover
$m(B)(1-\theta)\theta$ which implies that it rests at most
$m(B)(1-\theta)^2$ points that have not visited $B$ yet.

By induction we
conclude that for $|b|=nN$ the measure of points not
covered after taking all $2^{nN}$ pre-images is less than
$m(B)(1-\theta)^n\to 0$ when $n\to\infty$.
\end{proof}

The following corollary, whose proof is similar to that of
Corollary \ref{4.7} gives that the rate of mixing is exponential
and concludes the proof of Theorem \ref{Teo3}.
\begin{Cor}
 It holds that
$\lim_{n\to\infty} |m(f^{-n}(A)\cap B)-m(A)m(B)|=0$ exponentially fast.
\end{Cor}

\noindent
{\em  M. J. Pacifico}:
Instituto de Matem\'atica,
Universidade Federal do Rio de Janeiro,
C. P. 68.530, CEP 21.945-970,
Rio de Janeiro, RJ, Brazil.  \\
E-mail: pacifico@im.ufrj.br .\\

\noindent
{\em R. Markarian}:
Instituto de Matem\'atica y Estad\'{\i}stica (IMERL),
Facultad de Ingenier\'{\i}a,
Universidad de la Rep\'ublica,
CC30, CP 11300, Montevideo, Uruguay.\\
E-mail: roma@fing.edu.uy.\\

\noindent
{\em J. Vieitez}:
Regional Norte, Universidad de la Republica,
Rivera 1350, CP 50000, Salto, Uruguay.\\
E-mail: jvieitez@unorte.edu.uy.

\end{document}